\documentclass[11pt]{amsart}
\usepackage{geometry}                
\usepackage{graphicx}
\usepackage{amssymb}
\usepackage{epstopdf}
\usepackage[cmtip,arrow, all]{xy}
\title{Higher topological complexity of aspherical spaces}
\date{}                                           
\author{Michael Farber}
\address{School of Mathematical Sciences \\
Queen Mary, University of London\\
London, E1 4NS\\
United Kingdom}
\email{m.farber@qmul.ac.uk}
\thanks{Michael Farber was partially supported by a grant from the Leverhulme Foundation.}
\author{John Oprea}
\address{Department of Mathematics\\
Cleveland State University\\
Cleveland OH 44115  \\
U.S.A.}
\email{jfoprea@gmail.com}
\begin{document}

\theoremstyle{plain}
\newtheorem{theorem}{Theorem}[section]
\newtheorem{proposition}[theorem]{Proposition}
\newtheorem{lemma}[theorem]{Lemma}
\newtheorem{corollary}[theorem]{Corollary}
\newtheorem{definition}[theorem]{Definition}
\newtheorem{remark}[theorem]{Remark}
\newtheorem{example}[theorem]{Example}

\newcommand{\secref}[1]{\S\ref{#1}}
\newcommand{\thmref}[1]{Theorem~\ref{#1}}
\newcommand{\propref}[1]{Proposition~\ref{#1}}
\newcommand{\lemref}[1]{Lemma~\ref{#1}}
\newcommand{\corref}[1]{Corollary~\ref{#1}}
\newcommand{\conjref}[1]{Conjecture~\ref{#1}}
\newcommand{\properref}[1]{Properties~\ref{#1}}
\newcommand{\remref}[1]{Remark~\ref{#1}}
\newcommand{\examref}[1]{Example~\ref{#1}}
\newcommand{\introthmref}[1]{Theorem~\ref{#1}}
\newcommand{\introcorref}[1]{Corollary~\ref{#1}}
\newcommand{\introprobref}[1]{Problem~\ref{#1}}

\def\CP{\mathbb C{\rm P}}
\def\RP{\mathbb R{\rm P}}
\def\HP{\mathbb H{\rm P}}

\newcommand {\Ext}{{\rm {Ext}}}
\newcommand {\Hom}{{\rm {Hom}}}
\newcommand{\C}{{{\mathcal C}}}
\newcommand{\R}{{\mathbf R}}

\newcommand{\Q}{{\mathbf Q}}

\newcommand{\D}{{\mathcal D}}
\newcommand{\Z}{{\mathbb Z}}
\newcommand{\uz}{{\underline \Z}}
\newcommand{\uu}{{\mathfrak u}}
\newcommand{\vv}{{\mathfrak v}}
\newcommand{\od}{{\mathcal {O_D}}}
\newcommand{\odd}{{\mathcal {O_{D'}}}}
\newcommand{\ui}{{\underline I}}
\newcommand{\um}{{\underline M}}
\newcommand{\zpp}{{\Z[\pi\times\pi]}}
\newcommand{\zp}{{\Z[\pi]}}
\newcommand{\uc}{{\underline C}}
\newcommand{\tc}{{\sf {TC}}}
\newcommand{\secat}{{\sf {secat}}}
\newcommand{\uP}{{\underline P}}
\newcommand{\cat} {{\sf {cat}}}
\newcommand{\cd}{{\rm {cd}}}
\newcommand{\F}{{\mathcal {F}}}
\newcommand{\pd}{{\rm {pd}}}
\newcommand{\un}{\underline N}
\newcommand{\ualpha}{\underline \alpha}
\newcommand{\cD}{\mathcal D}
\newcommand{\TC}{{\sf {TC}}}

\begin{abstract}
In this article we study the higher topological complexity $\tc_r(X)$ in the case when $X$ is an aspherical space, $X=K(\pi, 1)$ and  $r\ge 2$.
We give a characterisation of $\tc_r(K(\pi, 1))$ in terms of classifying spaces for equivariant Bredon cohomology.
Our recent paper \cite{FGLO}, joint with M. Grant and G. Lupton, treats the special case $r=2$. We also obtain in this paper useful lower bounds for $\tc_r(\pi)$ in terms of cohomological dimension of subgroups of $\pi\times\pi\times \dots\times \pi$ ($r$ times) with certain properties.
As an illustration of the main technique we find the higher topological complexity of the Higman's groups. We also apply our method
to obtain a lower bound for the higher topological complexity of the right angled Artin (RAA) groups, which, as was
 established in \cite{GGY} by a different method (in a more general situation), coincides with the precise value.
 We finish the paper by a discussion of the $\tc$-generating function $\sum_{r=1}^\infty \tc_{r+1}(X)x^r$ encoding the values of the
 higher topological complexity $\tc_r(X)$ for all values of $r$.
We show that in many examples (including the case when $X=K(H, 1)$ with $H$ being a RAA group) the $\tc$-generating function is a rational function of the form $\frac{P(x)}{(1-x)^2}$ where $P(x)$ is an integer polynomial with $P(1)=\cat(X)$.

\end{abstract}

\maketitle

\section{Introduction and statement of the result}

\subsection{}

Suppose that a mechanical system has to be programmed to move autonomously from any initial state to any final state.
Let $X$ denote the configuration space of the system; points of $X$ represent states of the system and
continuous paths
in $X$ represent
motions of the system.
{\it A motion planning algorithm} is a function which associates with any pair of states $(A, B)\in X\times X$ a continuous motion of the system
starting at $A$ and ending at $B$. In other words, a motion planning algorithm is
a section of the path fibration
\begin{eqnarray}\label{fibration}
p: X^I\to X\times X, \quad p(\gamma) = (\gamma(0), \gamma(1)).
\end{eqnarray}
Here $X^I$ denotes the space of all continuous paths $\gamma: I=[0, 1]\to X$, equipped with the compact-open topology.
Unfortunately, a global motion planning algorithm is impossible to achieve unless the configuration space $X$ is contractible.
If $X$ is not contractible, then only ``local'' motion plans may be found.

{\it The topological complexity}, $\tc(X)$, informally, is the minimal number of continuous rules (i.e. local planners) which are needed to construct an
algorithm for autonomous motion planning of a system having $X$ as its configuration space. The quantity $\tc(X)$, originally introduced in \cite{Far03}
(see also \cite{Finv}), is, in fact, a numerical homotopy invariant of a path-connected topological space $X$ and so may be studied with all the tools
of algebraic topology. A recent survey of the concept $\tc(X)$ and robot motion planning algorithms in practically interesting configuration spaces can be found in \cite{Frecent}.

\subsection{The concept of higher or sequential topological complexity} Yuli Rudyak \cite{Rud} introduced a generalisation of the notion of topological complexity $\tc(X)$ which is usually denoted
 $\tc_r(X)$ and is called {\it the higher} or {\it sequential} topological complexity. 
 Here $X$ is a path-connected topological space and $r\ge 2$ is an integer.
The number $\tc_2(X)$ coincides with $\tc(X)$.
To define $\tc_r(X)$ consider the fibration
\begin{eqnarray}\label{rfibration}
p_r: X^I\to X^r,
\end{eqnarray}
where
\begin{eqnarray}
p_r(\gamma) = (\gamma(0), \gamma\left(\frac{1}{r-1}\right), \gamma\left(\frac{2}{r-1}\right), \dots, \gamma\left(\frac{r-2}{r-1}\right), \gamma(1)), \quad \gamma\in X^I.\end{eqnarray}
As above, $X^I$ denotes the space of all continuous paths $\gamma: I=[0, 1]\to X$ equipped with the compact-open topology.
The notation $X^r$ denotes $X\times X\times \dots\times X$ ($r$ times), the Cartesian product of $r$ copies of $X$.
The map $p_r$ associates with a path $\gamma: I\to X$ in $X$ the sequence of its locations at $r$ points $\frac{i}{r-1}\in I$ where $i=0, 1, \dots, r-1$
includes the initial and final states $\gamma(0), \gamma(1)$ and $r-2$ intermediate points.

\begin{definition}\label{rdef1}
Given a path-connected topological space $X$, {\it the $r$-th sequential topological complexity} of $X$ is defined as the minimal integer
$\tc_r(X)=k$ such that the Cartesian power $X^r$ can be covered by $k+1$ open subsets
$$X^r = U_0 \cup U_1 \cup \dots U_k$$ with the property that for any $i = 0, 1,2,...,k$ there exists a continuous section
$$s_i: U_i \to X^I,\quad
p_r\circ s_i={\rm id}$$ of the fibration (\ref{rfibration}) over $U_i$. If no such $k$ exists we will set $\tc_r(X)=\infty$.
\end{definition}

In other words, $\tc_r(X)$ is the Schwarz genus (or sectional category) of fibration (\ref{rfibration}), see \cite{Sv66}.

\subsection{} The invariant $\tc_r(X)$ has a clear meaning for the motion planning problem of robotics.
Assume that a system (robot) has to be programmed to move autonomously from any initial state to any final state such that it visits $r-2$ additional states on the way. If $X$ denotes the configuration space of the system then $\tc_r(X)$ is the minimal number of continuous rules needed to program the robot to
perform autonomously the indicated task. We note here that the most basic estimate for $\tc_r$ is in terms of the Lusternik-Schnirelmann category
(see \cite{CLOT}):
\begin{eqnarray}\label{basicest}
\cat(X^{r-1}) \leq \tc_r(X) \leq \cat(X^r) \leq r\,\cat(X).
\end{eqnarray}
We won't require any results about category here except for the fact that $\cat(K(\pi,1)) = \cd(\pi)$, where $\cd(\pi)$ is the
cohomological dimension of the group $\pi$. Here the symbol $K(\pi, 1)$ denotes the Eilenberg - MacLane space with the properties
$\pi_i(K(\pi, 1))=0$ for $i\not=1$ and $\pi_1(K(\pi, 1))=\pi.$
Some recent results concerning the invariant $\tc_r(X)$ with $r\ge 2$ can be found in  \cite{BGRT}, \cite{GG}, \cite{GGY}, \cite{Yuz}.
Also, the introduction to \cite{FGLO} gives an account of most of the significant recent developments regarding $\tc(X)$.

\subsection{}
One of the main properties of $\tc_r(X)$ is its homotopy invariance, which means that $\tc_r(X)=\tc_r(Y)$ if $X$ and $Y$ are homotopy equivalent.
In particular we obtain that in the case when $X=K(\pi, 1)$ the number $\tc_r(X)$ is an algebraic invariant of the group $\pi$.
We shall introduce the notation $$\tc_r(\pi) = \tc_r(K(\pi, 1)).$$

Our aim in this paper is to give a characterisation of $\tc_r(\pi)$ in terms of equivariant topology. On the face of it, there seems to be no connection between
these topics, but in our main result Theorem \ref{thm1} we will describe a path that unites them. In fact, Theorem \ref{thm1} generalizes a
similar connection that was displayed in \cite{FGLO} for $\tc$ itself. The emergence of equivariant topology as a player in the study of $\tc$ allows
invariants of the subject such as Bredon cohomology (with respect to a family of subgroups) to be used to estimate $\tc$. Moreover, as a result of
Theorem \ref{thm1}, new and interesting lower bounds are obtained for $\tc_r$ as in Theorem \ref{GLO} below. In previous works, lower bounds for $\tc_r$
tended to arise from cohomology (as ``cuplength''-type calculations). Here, however, Theorem \ref{GLO} gives a lower bound that is more
intrinsic to the subgroup structure of $\pi$. This type of result can be applied even when cuplength structure is missing (as in Theorem \ref{higman} below).

\section{A lower bound for $\tc_r(\pi)$}

\subsection{} In this subsection we shall state a useful corollary of our main result (Theorem \ref{thm1}) which gives a lower bound for $\tc_r(\pi)$; it has the advantage of
being stated using very simple algebraic terms.

Fix an integer $r\ge 2$ and consider the $r$-fold Cartesian product $$\pi^r =\pi\times\pi\times\dots\times \pi.$$
We shall consider the diagonal subgroup
\begin{eqnarray}\label{delta}
\Delta=\{(g, g, \dots, g); g\in \pi\}\subset \pi^r.
\end{eqnarray}
 Any subgroup $H\subset \pi^r$ conjugate to the diagonal $\Delta$ has the form $H=c^{-1}\Delta c$ where $c=(c_1, c_2, \dots, c_r)\in \pi^r$.
Consider a subgroup $K\subset \pi^r$  with the property that
$K\cap H=\{1\}$ for any subgroup $H\subset \pi^r$ conjugate to $\Delta$. This property of $K$ can be characterised as follows.
For an element $g\in \pi$ we shall denote by $[g]\subset \pi$ its conjugacy class. Then for any non-unit element
$1\not=(g_1, g_2, \dots, g_r)\in K$ there exists $i, j\in \{1, 2, \dots, r\}$ such that $[g_i]\not=[g_j]$.

\begin{theorem}\label{GLO}
Let $\pi$ be a discrete group, $r\ge 2$ an integer, and let $K\subset \pi^r=\pi\times \pi\dots \times\pi$ be a subgroup with the property that
$K\cap H=1$ for any subgroup $H\subset \pi^r$ conjugate to the diagonal $\Delta\subset \pi^r$. Then
\begin{eqnarray}
\tc_r(\pi)\ge \cd(K),
\end{eqnarray}
where $\cd(K)$ denotes the cohomological dimension of $K$.
\end{theorem}
Theorem \ref{GLO} generalises Corollary 3.5.4 from \cite{FGLO} (where the case $r=2$ was covered) as well as the result of
\cite{GrantLuptonOprea} where the class of subgroups of type $K=A\times B$ is considered assuming that $r=2$.

Theorem \ref{GLO} may be applied in the following situations.
In the special case when the subgroup $K\subset \pi^r$ has the form
$K=A_1\times A_2\times \dots\times A_r$
with $A_i\subset \pi$, the assumption of Theorem \ref{GLO} requires that for any collection $g_1,  \dots, g_r\in \pi$
the intersection
$$\bigcap_{i=1}^r g_iA_i g_i^{-1}=\{1\}$$ is trivial. In particular, we may take $K=A\times B\times \pi\times \dots\times \pi$ with the subgroups $A, B\subset \pi$ satisfying
$A\cap gBg^{-1}=\{1\}$ for any $g\in \pi$ as in  \cite{GrantLuptonOprea}.

One may always apply Theorem \ref{GLO} with $K=1\times \pi^{r-1}\subset \pi^r$ which gives the well-known inequality
$$\tc_r(\pi)\ge \cd(\pi^{r-1}).$$

For the free abelian group $\pi=\Z^k$ one has $\tc_r(\pi)=(r-1)k=\cd(\pi^{r-1})$, i.e. in this case the above inequality is sharp.


We shall use Theorem \ref{GLO} to prove the following:

\begin{theorem}\label{higman}
Let $\mathcal{H}$ denote Higman's group with presentation
$$P=\langle x , y, z, w \mid xyx^{-1}y^{-2},\, yzy^{-1}z^{-2},\, zwz^{-1}w^{-2},\, wxw^{-1}x^{-2} \rangle.$$
Then
$$\TC_r(\mathcal{H})=2r \quad \mbox{for any}\quad r\ge 2.$$
\end{theorem}

The proof of Theorem \ref{GLO} is given in \S \ref{prthglo}. The proof of Theorem \ref{higman} is given in \S \ref{sec:higman}.
\section{Statement of the main result}

\subsection{} To state the main result we need to recall a few standard notions; we shall mainly follow \cite{Lue}.

Let $G$ be a discrete group. {\it A $G$-CW-complex} is a CW-complex $X$ with a left $G$-action such that for each open cell
$e\subset X$ and each $g\in G$ with $ge\cap e\not=\emptyset$, the left multiplication by $g$ acts identically on $e$.

{\it A family $\F$ of subgroups} of $G$ is a set of subgroups of $G$ which is closed under conjugation and finite intersections.

\subsection{} \label{r222}{\it A classifying $G$-CW-complex} $E_\F(G)$ with respect to a family $\F$ is defined as a $G$-CW-complex $E_\F(G)$ such that
\begin{itemize}\item[{(a)}] the isotropy subgroup of any element of $E_\F(G)$ belongs to $\F$; \item[{(b)}] For any $G$-CW-complex $Y$ whose all
isotropy subgroups belong to $\F$ there is up to $G$-homotopy
exactly one $G$-map $Y\to E_F(G)$.
\end{itemize}

A $G$-CW-complex $X$ is a model for $E_\F(G)$ if and only if all its isotropy subgroups belong to the family $\F$ and for each $H\in \F$ the set
of $H$-fixed points $X^H$ is weakly contractible, i.e. $\pi_i(X^H, x_0)=0$ for any $i=0, 1, \dots$ and for any $x_0\in X^H$. See \cite{Lue}, Theorem 1.9.

\subsection{} We shall assume below that a discrete group $\pi$ is fixed.
Let $G$ denote the group $\pi^r$ where $r\ge 2$. Let $\Delta\subset \pi^r$ denotes the diagonal subgroup (\ref{delta}).
Let $\D$ be the minimal family of subgroups of $G$ containing the diagonal subgroup
$\Delta\subset \pi^r$ and the trivial subgroup, which is closed under conjugations by elements of $\pi^r$ and under taking finite intersections.

We shall consider the classifying spaces $E(G)$ and $E_\D(G)$, where $E(G)$ is the classical classifying space for free actions.
The universal properties of classifying spaces
imply the existence of a
$G$-map
\begin{eqnarray}\label{can1}
f: E(G) \to E_\D(G),
\end{eqnarray}
which is unique up to equivariant homotopy.

Now we may state the main result of this paper:

\begin{theorem}\label{thm1}
Let $X$ be a finite aspherical cell complex, let $\pi=\pi_1(X, x_0)$ be its fundamental group and let $r\ge 2$ be an integer.
Denote $G=\pi^r=\pi\times\dots\times \pi$. Then the topological complexity $\tc_r(X)=\tc_r(\pi)$ coincides with the smallest integer $k$ such that the canonical map (\ref{can1}) can be factorised (up to $G$-equivariant homotopy) as
\begin{eqnarray}\label{rFF}
E(G)\to L\to E_\D(G)
\end{eqnarray}
where $L$ is a $G$-CW-complex of dimension $\le k$. Here $\D$ is the family of subgroups of $G$ defined above. Equivalently, $\tc_r(X)$ coincides with the smallest integer $k$ such that the canonical map (\ref{can1}) is $G$-equivariantly homotopic to a map taking values in the $k$-dimensional skeleton
$$E_\D(G)^{(k)}\subset E_\D(G).$$
\end{theorem}

Theorem \ref{thm1} and its proof generalise the results of \cite{FGLO} where the case $r=2$ was treated.

It is known that the classifying space $E_\D(G)$ admits a realisation as a $G$-CW-complex of dimension $\max\{3, \cd_\D(G)\}$, see \cite{Lue}.
Here the symbol $\cd_\D(G)$ stands for the cohomological dimension of the trivial  $O_\D$-module $\uz$ and $O_\D$ denotes the orbit category with orbits of type $\D$, see \cite{Lue} or \cite{Mis}.
Hence we obtain the following corollary:

\begin{corollary}
One has
\begin{eqnarray}
\tc_r(\pi) \le \max\{3, \cd_\D(\pi^r)\}, \quad r\ge 2.
\end{eqnarray}
\end{corollary}
Later in this paper we shall supplement this upper bound on $\tc_r(\pi)$ by lower bounds based on Bredon cohomology.

\subsection{Proof of Theorem \ref{GLO}}\label{prthglo} In this subsection we show how Theorem \ref{GLO} follows from Theorem \ref{thm1}.
Denote $G=\pi^r$ and consider a decomposition
$$E(G)\stackrel{\alpha}\to L\stackrel{\beta}\to E_\D(G),$$ where $L$ is a $G$-CW-complex of dimension $\dim L=\tc_r(\pi)$, as given by Theorem \ref{thm1}; here $\alpha$ and $\beta$ are $G$-equivariant maps.
The subgroup $K\subset G$ acts freely on $E(G)$ and on $E_\D(G)$ - here we use our assumption on the subgroup $K$. Hence
both spaces $E(G)$ and $E_\D(G)$ can be viewed as models of $E(K)$. We may fix $K$-equivariant homotopy equivalences
$a: E(K)\to E(G)$ and $b:E(K)\to E_\D(G)$. For any $G$-equivariant map $f:E(G)\to E_\D(G)$ (as (\ref{can1})) we
have $$f\circ a \, \simeq_K \, b$$ as follows from the universal property of $E(K)$; here the sign $\simeq_K$ denotes a $K$-equivariant homotopy.
Taking $f= \beta\circ \alpha$, we see that for any $\Z[K]$-module $M$ and for any cohomology class
$\gamma\in H^i(E(K), M)=H^i(K, M)$, with $$i>\tc_r(\pi)=\dim L,$$
we may write $\gamma= b^\ast(\gamma')$ where $\gamma'\in H^i(E_\D(G), M)$ and for obvious reasons
the class $\alpha^\ast\beta^\ast(\gamma')\in H^i(E(G), M)$ is trivial. This implies that $$\gamma=b^\ast(\gamma')=a^\ast f^\ast(\gamma')=a^\ast\alpha^\ast\beta^\ast(\gamma')=0,$$ i.e.
 $H^i(K, M)=0$ for any $i>\tc_r(\pi)$ and for any $\Z[K]$-module $M$. This proves that $\cd(K)\le \tc_r(\pi)$ as claimed. \qed

\section{Proof of Theorem \ref{thm1}}

\subsection{The invariant $\tc_r^{\D}(X)$} We shall use a convenient modification of the concept $\tc_r(X)$.
As before we denote by $\D$ the family of subgroups of $G=\pi^r$ generated by the diagonal $\Delta$.

\begin{definition}\label{rdef2}
Let $X$ be a path-connected topological space with fundamental group $\pi=\pi_1(X, x_0)$ and $r\ge 2$ an integer. {\it The $\D$-topological complexity}, $\tc^\D_r(X)$, is defined as the minimal number
$k$ such that $X^r$ can be covered by $k+1$ open subsets
$$X^r = U_0 \cup U_1 \cup \dots U_k$$ with the property that for any $i = 0, 1,2,...,k$
and for any choice of the base point $u_i\in U_i$ the homomorphism $\pi_1(U_i, u_i)\to \pi_1(X^r, u_i)$ induced by
the inclusion $U_i\to X^r$ takes values in a subgroup conjugate to the diagonal $\Delta\subset \pi^r$.
\end{definition}

To ensure that  Definition \ref{rdef2} makes sense,  recall that for any choice of the base point $u=(u_1, \dots, u_r)\in X^r$ one has
$$\pi_1(X^r, u)=\pi_1(X, u_1)\times \pi_1(X, u_2)\times \dots\times \pi_1(X, u_r)$$
and
there is an isomorphism $\pi_1(X^r, u) \to \pi_1(X^r, (x_0, \dots, x_0))=\pi^r$ determined uniquely up to conjugation. Moreover,
the diagonal inclusion $X\to X^r$, $\, $ $x\mapsto (x, x, \dots, x)$, induces the inclusion $\pi\to\pi^r$ onto the diagonal $\Delta$.

\begin{lemma}\label{lm3}
If $X$ is a finite aspherical cell complex then $\tc^\D_r(X)=\tc_r(X)$.
\end{lemma}
\begin{proof}
Consider an open subset $U\subset X^r$ and
a continuous section  $s: U\to X^I$ of the fibration (\ref{rfibration}) over $U$.
Using the exponential correspondence, the map $s$ can be viewed as a homotopy $h:U\times I\to X$ where $h(x, t)=s(x)(t)$ for $x\in U, t\in I$.
One has
$$h\left(x, \frac{i-1}{r-1}\right) = x_i\quad\mbox{for}\quad i=1, \dots, r$$
where $x=(x_1, \dots, x_r)$.
Let $q_j: X^r\to X$ (where $j=1, \dots, r$) denote the projection onto the $j$-th factor.
The property of $s$ to be a section of (\ref{rfibration}) can be expressed by saying that the homotopy $h| U\times [\frac{i-1}{r}, \frac{i}{r}]$ connects the projections
$q_{i}: U\to X$ and $q_{i+1}: U\to X$
for $j=1, \dots, r$.

Thus we see that the open sets $U_i\subset X\times X$ which appear in Definition \ref{rdef1} can be equivalently characterised by the property that their projections $q_j: U_i\to X$ on all the factors $j=1, \dots, r$ are homotopic to each other.

Since $X$ is aspherical,
for any connected space $U$, which is homotopy equivalent to a cell complex, the set of homotopy classes of maps $U\to X$ is in one-to-one correspondence
with the set of conjugacy classes of homomorphisms $\pi_1(U, u)\to \pi_1(X, x_0)$, see \cite{Whi}, Chapter V, Corollary 4.4.
Recall that an open subset of a CW-complex is an ANR and therefore is homotopy equivalent to a countable CW-complex, see Theorem 1 in \cite{Mil}.
Thus we see that
an open subset $U\subset X^r$ admits a continuous section of fibration (\ref{fibration}) if and only if the induced homomorphisms
$${q_j}_\ast: \pi_1(U_1, u)\to (X, u_j), \quad j=1, \dots, r$$
are conjugate to each other. Here $u=(u_1, \dots, u_r)\in U$ is a base point.  This latter condition is obviously equivalent to the requirement
(which appears in Definition \ref{rdef2}) that the map on $\pi_1$ induced by the inclusion
$U\to X^r$ takes values in a subgroup conjugate to the diagonal $\Delta$. This completes the proof.
\end{proof}

\begin{corollary}\label{rcor4}
Let $X$ be a connected finite aspherical cell complex with fundamental group $\pi=\pi_1(X, x_0)$ and $r\ge 2$.
Let
$q: \widehat{X^r}\to X^r$ be the connected covering space
corresponding to the diagonal subgroup $$ \Delta\subset \pi^r=\pi_1(X^r, (x_0, \dots, x_0)).$$
Then
the $\D$-topological complexity $\tc^\D_r(X)$ coincides with the Schwarz genus of $q$.
\end{corollary}
\begin{proof}
For an open subset $U\subset X^r$, the condition that the induced map $\pi_1(U, u) \to \pi_1(X^r, u)$ takes values in a subgroup
conjugate to the diagonal $\Delta$ is equivalent to the condition that $q$ admits a continuous section over $U$.
The Lemma follows by comparing the definitions of $\tc^\D_r(X)$ and of Schwarz genus.
\end{proof}


\subsection{} \label{sec22} The covering $q: \widehat{X^r}\to X^r$ which appears in Corollary \ref{rcor4} is a regular covering only when $\pi$ is abelian. This covering can be characterised by the property that the image of the homomorphism
$q_\ast: \pi_1(\widehat{X^r})\to \pi_1(X^r)$ is a subgroup conjugate to $\Delta\subset \pi^r$. Below we describe this covering in more detail.

\subsection{} Let $\pi$ be a discrete group and $r\ge 2$. Denote $G=\pi^r$.
We shall view $$\pi^{r-1}=\pi\times\dots\times \pi$$  (the Cartesian product of $r-1$ copies of $\pi$) as a discrete topological space with the following left $G=\pi^r$-action:
\begin{eqnarray}\label{raction}
(x_0, x_1, \dots, x_r) \cdot (g_1, \dots, g_r) = (x_0g_1x_1^{-1}, x_1g_2x_2^{-1}, \dots, x_{r-1}g_rx_r^{-1}),
\end{eqnarray}
where $(x_0, \dots, x_r)\in \pi^r$ and $(g_1, \dots, g_r)\in \pi^{r-1}.$ This action is transitive and the isotropy subgroup of the element $(1, 1, \dots, 1)\in \pi^{r-1}$ coincides with the diagonal subgroup
$\Delta\subset \pi^r=G$.
The isotropy subgroups of the other elements are the conjugates of $\Delta$.

Consider the universal covering $\tilde{X^r}\to X^r$. The space $\tilde{X^r}$ carries a free left $G$-action and we may consider $\tilde{X^r}\to X^r$ as a principal $G$-fibration. The associated fibration
\begin{eqnarray}\label{fib}\tilde{X^r}\times_G \pi^{r-1}\to X^r\end{eqnarray} (the Borel construction) coincides with the covering $q: \widehat{X^r}\to X^r$. Indeed,
since
$$\tilde{X^r}\times_G \pi^{r-1} = \tilde{X^r}\times_G (G/\Delta) = \tilde{X^r}/\Delta,$$
the fundamental group
of the space $\tilde{X^r}\times_G \pi^{r-1}$
can be naturally identified with $\Delta$  and therefore (\ref{fib}) coincides with the covering corresponding to $\Delta$. Thus, $\tc^\D_r(X)$ coincides with the Schwarz genus of the fibration (\ref{fib}).

\subsection{} The join $X\ast Y$ of topological spaces $X$ and $Y$ can be defined as the quotient of the product $X\times [0,1]\times Y$ with respect to the equivalence relation $(x, 0, y)\sim (x, 0, y')$ and $(x, 1, y)\sim (x', 1, y)$ for all $x, x'\in X$ and $y, y'\in Y$. We have an obvious embedding
$X\to X\ast Y$ given by $x\mapsto (x, 0, y)$ where $y\in Y$ is arbitrary.

A point $(x, t, y)\in X\times [0,1]\times Y/\sim$ \, can be written as a formal linear combination $(1-t)x+ty$. This notation is clearly consistent with the identifications of the join.

\subsection{} \label{sec24} For an integer $k\ge 0$, let $E_k(\pi^{r-1})$ denote the $(k+1)$-fold join
$$E_k(\pi^{r-1}) = \pi^{r-1}\ast\pi^{r-1}\ast \dots\ast \pi^{r-1}.$$
We shall equip $E_k(\pi^{r-1})$ with the left diagonal $G=\pi^r$-action determined by the $G$-action on $\pi^{r-1}$, see  (\ref{raction}).
Each $E_k(\pi^{r-1})$ is naturally a $k$-dimensional equivariant simplicial complex with $k$-dimensional simplexes in 1-1 correspondence with sequences $(g_0, g_1, \dots, g_k)$ of elements
$g_i\in \pi^{r-1}$.
Note also that $E_k(\pi^{r-1})$ is $(k-1)$-connected. 

\subsection{} Next we apply a theorem of A. Schwarz (see \cite{Sv66}, Theorem 3) stating that genus of a fibration $p: E\to B$ equals the smallest integer $k$ such that
the
fiberwise join $p\ast p\ast \dots \ast p$ of $k+1$ copies of $p: E\to B$ admits a continuous section. We apply this criterion to the fibration (\ref{fib}).
The fiberwise join of $k+1$ copies of (\ref{fib}) is obviously the fibration
\begin{eqnarray}\label{qk}
q_k: \, \tilde X^r \times_G E_k(\pi^{r-1}) \to X^r,
\end{eqnarray}
where the left $G$-action on $E_k(\pi^{r-1})$ is described above. Hence we obtain that the number $\tc_r^\D(X)$ coincides with the smallest $k$ such that
(\ref{qk}) admits a continuous section.

\subsection{} \label{r25}
Finally we apply Theorem 8.1 from \cite{Hue}, chapter 4,  which states that continuous sections of the fibre bundle $q_k$ are in 1-1 correspondence with
$G$-equivariant maps
\begin{eqnarray}\label{rfinally}\tilde X^r \to E_k(\pi^{r-1}).\end{eqnarray}
Thus, we see that $\tc_r^\D(X)$ is the smallest $k$ such that a $G$-equivariant map (\ref{rfinally}) exists.

\subsection{} In the rest of the proof we shall assume that $X$ is an aspherical finite cell complex. We observe that the space of the universal cover
$\tilde X^r$ is a contractible CW-complex with a free $G$-action, thus $\tilde X^r$ is a model of the classifying space $E(G)$.

\subsection{}\label{rsec15} There is a natural equivariant embedding $$E_k(\pi^{r-1}) \stackrel \subset \to E_{k+1}(\pi^{r-1})=E_k(\pi^{r-1})\ast \pi^{r-1}.$$ Using it we may define a $G$-CW- complex
$$E(\pi^{r-1})= \bigcup_{k=1}^\infty E_k(\pi^{r-1}) = \pi^{r-1}\ast \pi^{r-1}\ast \pi^{r-1}\ast \dots,$$
the join of infinitely many copies of $\pi^{r-1}$.
We claim that the $G$-complex $E(\pi^{r-1})$ is a model for the classifying space $E_\D(G)$.
Indeed, $E(\pi^{r-1})$ is a simplicial complex with a simplicial $G$-action hence a $G$-CW-complex (with respect to the barycentric subdivision),
see \cite{Lue}, Example 1.5.

We want to show that: (a) the isotropy subgroup of every point $x\in E(\pi^{r-1})$ belongs to the family $\D$ and (b)
that for any $H\in \D$
the fixed point set $E(\pi^{r-1})^H$ is contractible.
Any point $x\in E(\pi^{r-1})$ can be represented in the form
$$x=t_0x_0+t_1x_1+\dots+t_kx_k$$
where $t_i\in (0, 1]$, $x_i\in \pi^{r-1}$ and $t_0+t_1+\dots+t_k=1$. Then the isotropy subgroup of $x$ is the intersection of the isotropy subgroups of
$x_i\in \pi^{r-1}$ which are all conjugates of $\Delta$; thus the isotropy subgroup of $x$ is a member of the family $\D$.
If $H\in \D$ then the set $E(\pi^{r-1})^H$ coincides with the infinite join
$$(\pi^{r-1})^H\ast (\pi^{r-1})^H\ast (\pi^{r-1})^H\ast\dots $$
which is obviously contractible. We see that properties (a) and (b) are satisfied and therefore the space $E(\pi^{r-1})$ is a model of the classifying space $E_\D(G)$.

\subsection{} Now we shall use the main properties of classifying spaces $E(G)$ and $E_\D(G)$, see \S \ref{r222}.
In particular any $G$-CW-complex $Y$ with isotropy in class $\D$
admits a unique up to $G$-homotopy $G$-map $Y\to E_\D(G)$. In particular, there exists unique up to homotopy maps
$$f: E(G)\to E_\D(G), \quad\mbox{and}\quad g: E_k(\pi^{r-1})\to E_\D(G).$$
One option for  $g$ is the natural inclusion $E_k(\pi^{r-1})\to E(\pi^{r-1})=E_\D(G)$. The map $f$ can be realised as follows.
Let $\phi: \pi^r\to \pi^{r-1}$ be given by
\begin{eqnarray}\label{phi}\phi(x_0, x_1, \dots, x_r)=(y_1, y_2, \dots, y_r), \quad \mbox{where}\quad y_i = x_{i-1}x_{i}^{-1}
\end{eqnarray} where
for $ i=1, \dots r$. It is easy to see that $\phi$ is $G$-equivariant. The natural extension of $\phi$ to the infinite joins
defines a $G$-equivariant map
$$
f: E(G)=\pi^r\ast \pi^r\ast \pi^r\ast\dots  \, \, \to \, \,  E_\D(G)= E(\pi^{r-1})=\pi^{r-1}\ast \pi^{r-1}\ast \pi^{r-1}\ast\dots.
$$

\subsection{} We have shown above (see Lemma \ref{lm3} and \S \ref{r25}) that $\tc_r(\pi)$ coincides with the smallest $k$ such that there exists a $G$-equivariant map
$h: E(G)\to E_k(\pi^{r-1})$. Composing with $g$ (see above) we obtain a $G$-map $g\circ h: E(G)\to E_\D(G)$ which must be $G$-homotopic to $f$.
We see that for $k=\tc_r(\pi)$
the map (\ref{can1}) factorises (up to $G$-homotopy) as
$$E(G)\to E_k(\pi^{r-1})\to E_\D(G),$$
where $\dim E_k(\pi^{r-1})=k$.


On the other hand, suppose  that the map (\ref{can1}) factorises as follows
$$E(G)=E(\pi^r) \stackrel\alpha\to L \stackrel \beta\to E_\D(G)=E(\pi^{r-1})$$
with $\dim L\le d$.
We want to apply the equivariant Whitehead Theorem (see Theorem \ref{thmwhite} below) to the inclusion
$E_d(\pi^{r-1})\to E(\pi^{r-1})$. For any subgroup $H\subset G$ we have
$$E_d(\pi^{r-1})^H= (\pi^{r-1})^H\ast (\pi^{r-1})^H\ast \dots \ast (\pi^{r-1})^H,$$
with $d+1$  factors. Thus,  $E_d(\pi^{r-1})^H$ is $d-1$-connected. Besides, $E(\pi^{r-1})^H$ is contractible. The Whitehead Theorem applied to
$\beta: L \to E(\pi^{r-1})$ gives a $G$-map
$\gamma: L\to E_d(\pi^{r-1})$. Composing with $\alpha$ we obtain  $\gamma\circ\alpha: E(G)\to E_d(\pi^{r-1})$; thus, using the results of Lemma \ref{lm3} and \S \ref{r25},
we obtain that $\tc_r(\pi) \le d$. This completes the proof. \qed

\begin{theorem}[Whitehead theorem, see \cite{May}, Theorem 3.2 in Chapter 1]\label{thmwhite} Let $f: Y\to Z$ be a $G$-map between $G$-CW-complexes such that for each subgroup $H\subset G$ the induced map
$\pi_i(Y^H, x_0)\to \pi_i(Z^H, f(x_0))$ is an isomorphism for $i<k$ and an epimorphism for $i=k$ for any base point $x_0\in Y^H$. Then for any $G$-CW-complex $L$ the induced map
on the set of $G$-homotopy classes
$$f_\ast: [L, Y]_G \to [L, Z]_G $$
is an isomorphism if $\dim L<k$ and an epimorphism if $\dim L\le k$.
\end{theorem}

\section{Lower bounds for $\tc_r(\pi)$ via Bredon cohomology}

In the theory of Lusternik - Schnirelmann category the following result plays an important role. If $X=K(\pi, 1)$ is an aspherical space and
$H^n(X;M)\not=0$ for some local coefficient system then $\cat(X)\ge n$, see \cite{EG}, \cite{Sv66}.
A word-to-word generalisation of this result for $\tc(X)$ fails as we have many examples of aspherical spaces $X=K(\pi, 1)$ such that $H^n(X\times X,M)\not=0$ while $\tc(X)<n$.

In \cite{FM} a notion of {\it an essential cohomology class} was introduced and the existence of a nonzero essential class $\xi\in H^n(\pi\times \pi, M)$ implies $\tc(\pi)\ge n$.

In this paper we
generalise the approach of \cite{FGLO} of using Bredon cohomology to detect essential cohomology classes. Namely we show that for any $r\ge 2$
the existence of a nonzero cohomology class $\xi\in H^n(\pi^r, M)$ which can be extended to a Bredon cohomology class
$\underline \xi\in H^n_\D(\pi^r, \um)$ (with respect to the family $\D$ of subgroups of $\pi^r$ which was described earlier) implies that $\tc_r(\pi)\ge n$. For $r=2$ this was proven in \cite{FGLO}.

\subsection{Bredon cohomology} \label{sec31} First we recall the construction of Bredon cohomology, see for example \cite{Mis}.

As above, let $G$ denote the group $\pi^r$, where $r\ge 2$, and $\D$ denote the minimal family of subgroups of $\pi^r$ containing the diagonal
$\Delta\subset \pi^r=G$ and the trivial subgroup $\{1\}\subset G$ which is closed under conjugations and finite intersections.

The symbol $\od$ denotes {\it the orbit category} which has as objects transitive left $G$-sets with isotropy in $\D$
and as morphisms $G$-equivariant maps, see \cite{Bre}. Objects of the category $\od$ have the form $G/H$ where $H\in \D$.

 {\it A (right) $\od$-module} $\um$ is a contravariant functor on the category of orbits $\od$ with values in the category of abelian groups. Such a module
is determined by the abelian groups $\um(G/H)$ where $H\in \D$, and by a group homomorphism
$$\um(G/H) \to \um(G/H')$$ associated with any $G$-equivariant map $G/H' \to G/H$.

Let $X$ be a $G$-CW-complex such that the isotropy subgroup of every point $x\in X$ belongs to the family $\D$. For every subgroup $H\in \D$
we may consider the cell complex $X^H$ of $H$-fixed points and its cellular chain complex $C_\ast(X^H)$. A $G$-map $\phi: G/K\to G/L$ induces a cellular map $X^L\to X^K$ by mapping $x\in X^L$ to $gx\in X^K$ where $g$ is determined by the equation $\phi(K)=gL$ (thus $g^{-1}Kgx=x$ since $g^{-1}Kg\subset L$ and therefore $Kgx=gx$, i.e. $gx\in X^K$). Thus we see that the chain complexes $C_\ast(X^H)$, considered for all $H\in \D$, form a
chain complex of right $\od$-modules which will be denoted ${\underline C}_\ast(X)$; here
$${\underline C}_\ast(X)(G/H) = C_\ast(X^H).$$

Note that the complex  ${\underline C}_\ast(X)$ is free as a complex of $\od$-modules although the complex $C_\ast(X)$ might not be free.

There is an obvious augmentation
$\epsilon: {\underline C}_0(X)\to \uz$ which reduces to the usual augmentation $C_0(X^H)\to \Z$ on each subgroup $H\in \D$.

If $\um$ is a right $\od$-module, we may consider the cochain complex of $\od$-morphisms $\Hom_\od ({\underline C}_\ast(X), \um)$.
Its cohomology
\begin{eqnarray}
H_\D^\ast(X; \um) \, = \, H^\ast(\Hom_\od ({\underline C}_\ast(X), \um))
\end{eqnarray}
is {\it the Bredon equivariant cohomology of $X$ with coefficients in $\um$. }

Let $M$ denote the principal component of $\um$.  Reducing to the principal components we obtain a homomorphism of cochain complexes
$$\Hom_\od ({\underline C}_\ast(X), \um)\to \Hom_{\Z[G]} ({C}_\ast(X), M)$$
and the homomorphism on the cohomology groups
\begin{eqnarray}\label{red}
 H^i_\D(X; \um) \, \to \, H^i_G(X, M).
\end{eqnarray}

\subsection{} \label{rsec34} If the action of $G$ on $X$ is free then the homomorphism (\ref{red}) is an isomorphism and
$$H^i_\D(X; \um) \, \simeq  \, H^i(X/G, M)$$
where on the right we have the usual twisted cohomology. In particular we obtain
$$H^n_\D(E(\pi^r), \um) = H^n(\pi^r, M).$$

\subsection{} \label{rsec35} Suppose now that $X=E(\pi^{r-1})$, viewed as a left $G$-CW-complex, where $G=\pi^r$, see \S \ref{sec24}.
We know that $E(\pi^{r-1})$ is a model for the classifying space $E_\D(G)$ and the classifying complex $E_\D(G)$ is unique up to $G$-homotopy.  Hence we may use the notation
$$H^\ast_\D(E(\pi^{r-1}), \um) = H^\ast_\D(\pi^r, \um).$$
 We obtain
that the number $\cd_\D(\pi^r)$ coincides with the maximal integer $n$ such that
$$H^i_\D(\pi^r, \um) =0$$
for all $i>n$ and for all $\od$-modules $\um$.

\subsection{} Consider now the effect of the $\pi^r$-equivariant map $F: E(\pi^r) \to E(\pi^{r-1})$, see (\ref{rFF}). Note that any two equivariant maps
$E(\pi^r) \to E(\pi^{r-1})$ are equivariantly homotopic.
The induced map on Bredon cohomology
$$F^\ast: H^i_\D(E(\pi^{r-1}), \um) \to H^i_\D(E(\pi^r), \um)$$
together with the notations introduced in \S \ref{rsec34} and \S\ref{rsec35} produce a homomorphism
\begin{eqnarray}\label{Phi}
\Phi: H^i_\D(\pi^r, \um) \, \to \, H^i(\pi^r, M)
\end{eqnarray}
which relates the Bredon cohomology with the usual group cohomology of $\pi^r$.

Now we state a result which gives useful lower bounds for the topological complexity $\tc_r(X)$.

\begin{theorem}\label{lower}
Let $X$ be a finite aspherical cell complex with fundamental group $\pi$. Suppose that
for some $\od$-module $\um$
there exists a Bredon cohomology class
$$\underline \alpha\in H^n_\D(\pi^r, \um)$$
such that the class
$$\Phi(\underline \alpha)\not=0\in H^n(\pi^r, M)$$
is nonzero.
Here $M$ denotes the principal component of $\um$. Then $\tc_r(X) \ge n$.
\end{theorem}

\begin{proof} Suppose that $\tc(X)<n$. Then by Theorem \ref{thm1} the map $F: E(\pi^r) \to E_\D(\pi^{r})$ admits a factorisation
$$E(\pi^r) \to L\to E_\D(\pi^r)$$
where $L$ is a $G$-CW-complex of dimension less than $n$. Then the homomorphism
$$\Phi: H^n_\D(\pi^r, \um) \to H^n(\pi^r, M)$$
factorises as
$$\Phi: H^n_\D(\pi^r, \um) \to H^n_\D(L, \um) \to H^n(\pi^r, M)$$
and the middle group vanishes since $\dim L <n$. This contradicts our assumption that
$\Phi(\underline \alpha)\not=0$ for some $\underline \alpha\in H^n_\D(\pi^r, \um)$.
\end{proof}

\section{Proof of Theorem \ref{higman}: Higman's group}\label{sec:higman}

G.\ Higman gave an example of a $4$-generator, $4$-relator group with some remarkable properties.
First, form the group $H_{xy}$ with presentation
\[
\langle x,y \mid xyx^{-1}y^{-2} \rangle.
\]
This group is isomorphic to the Baumslag--Solitar group $B(1,2)$, and hence is a duality group of dimension $2$.

The infinite cyclic group $F(y)$ injects into both $H_{xy}$ and $H_{yz}$, and so we may form $$H_{xyz}:=H_{xy}\ast_{F(y)} H_{yz}.$$
We may also form $H_{zwx}$ as the amalgamated sum of $H_{zw}$ and $H_{wx}$ over $F(w)$. The free group $F(x,z)$ injects into
both $H_{xyz}$ and $H_{zwx}$, and Higman's group is defined to be
$$\mathcal{H}:=H_{xyz}\ast_{F(x,z)} H_{zwx}.$$ It has presentation
\[
\mathcal{H}= \langle x , y, z, w \mid xyx^{-1}y^{-2},\, yzy^{-1}z^{-2},\, zwz^{-1}w^{-2},\, wxw^{-1}x^{-2} \rangle.
\]
The group $\mathcal{H}$ is acyclic (it has the same integer homology as a trivial group), and so $\tilde{H}^\ast(\mathcal{H};k)=0$ for every abelian
group $k$. Moreover,  it has no non-trivial finite dimensional representations over any field and so if $M$ is any coefficient $\Z[\mathcal{H}]$-module which is
finitely generated as an abelian group, then $\tilde{H}^\ast(\mathcal{H};M)=0$. Thus the group $\mathcal{H}$ is difficult to distinguish from a trivial group
using cohomological invariants. On the other hand, since $\mathcal{H}$ is not a free group so we have $\cd(\mathcal{H})\ge 2$. The $2$-dimensional
complex associated to the presentation of $\mathcal H$ given above is aspherical and it follows that
$$\cat(\mathcal{H})=\cd(\mathcal{H})=\dim(K(\mathcal{H},1)=2$$
where, by $\dim$ we refer to the smallest dimension of a $K(H,1)$ complex. Thus the higher topological complexity
of Higman's group satisfies $\TC_r(\mathcal{H})\leq 2r$, using the general result that
$$\TC_r(X) \leq \cat(X^r) \leq \dim(X^r) \leq r\,\dim(X).$$
Note that the zero-divisors cup length over any field is zero, so lower cohomological bounds are hard to come by. In \cite{GrantLuptonOprea} using a
geometric group theory argument due to Yves de Cornulier, it was shown that $gH_{xy}g^{-1}\cap H_{zw}=\{1\}$ for all $g\in \mathcal{H}$. Furthermore,
both of these groups are isomorphic to the Baumslag--Solitar group $B(1,2)$, hence are duality groups of dimension $2$. Thus the product
$$K=H_{xy}\times H_{zw} \times \cdots \times H_{zw}$$
(with $r-1$ factors $H_{zw}$) is a duality group of dimension $2r$ and so $\cd(H_{xy}\times H_{zw}^{r-1})=2r$.
We may apply
Theorem \ref{GLO} to obtain
$$2r = \cd(H_{xy}\times H_{zw}^{r-1}) \leq \TC_r(\mathcal{H}) \leq 2r$$
so that $\TC_r(\mathcal{H})=2r$. \qed

\section{The higher topological complexity of right angled Artin groups}

\subsection{} Let $\Gamma=(V, E)$ be a finite graph and let $H=H_\Gamma$ be the right angled Artin (RAA) group associated to $\Gamma$.
Recall that $H= H_\Gamma$ is given by a presentation with generators $v\in V$ and relations $vw=wv$, for each edge $(v, w)\in E$.
In Theorem \ref{GGY} below we state the result of \cite{GGY} which computes the topological complexity $\tc_r(H_\Gamma)$. Our goal here
is to give a new vastly simplified proof of the relevant lower bound using Theorem \ref{GLO}. An upper bound implying the equality requires finding explicit motion planners
and this may be found in \cite{GGY}. We shall need the following definition.

\begin{definition}
For a graph $\Gamma=(V, E)$ and for an integer $r\ge 2$ we define the number $z_r(\Gamma)$ as the maximal total cardinality $\sum_{i=1}^r |C_i|$
of $r$ cliques $C_1, \dots, C_r\subset V$ with empty intersection, $\cap_{i=1}^r C_i=\emptyset$.
\end{definition}
Recall that a clique of a graph $\Gamma=(V, E)$ is a set of vertices $C\subset V$ such that any two are connected by an edge.
In other words, a clique is a complete induced subgraph of $\Gamma$.

One may equivalently define $z_r(\Gamma) = \max\left\{\sum_{i=1}^r |C_i| - |\cap_{i=1}^r C_i |\right\}$ where
 $C_1,\ldots,C_r$ run over all sequences of $r$ cliques in $\Gamma$.
Since our original definition
is included in this one the only question is whether the new definition can give a
strictly greater number. To see that this cannot happen we note that given an arbitrary sequence of $r$ cliques
$C_1, \dots, C_r$ we may modify it by subtracting from the last clique $C_r$ the intersection $\cap_{i=1}^r C_i$ obtaining a sequence as in
the original definition with the same value of the total sum.
%

\begin{theorem}\label{GGY}\cite{GGY} One has $\tc_r(H_\Gamma) = z_r(\Gamma)$.
\end{theorem}

In \cite{GGY} the main result is stated slightly differently since the authors operate in higher generality and use a different language.
However it easy to see that Theorem \ref{GGY} follows from Theorem 2.7 and Proposition 2.3 in \cite{GGY}. Below we shall see that
the lower bound $z_r(\Gamma) \leq \tc_r(H_\Gamma)$ follows directly and easily from simple results about cliques and Theorem \ref{GLO}.

We first observe that, if $c(\Gamma)$ denotes the size of the maximal clique, then
\begin{eqnarray}\label{eqnest}
(r-1)c(\Gamma)\le z_r(\Gamma)\le rc(\Gamma).
\end{eqnarray}
The right inequality follows from $|C_i|\le c(\Gamma)$. The right inequality can be strict if the graph $\Gamma$ contains $r$ cliques
of maximal size $c(\Gamma)$ with disjoint intersection.  To prove the left inequality we note that
we may always take $C_1=\dots=C_{r-1}$ of size $c(\Gamma)$ and $C_r=\emptyset$.
The estimates given by (\ref{eqnest}) are in fact the algebraic analogue of the topological estimates in (\ref{basicest}).


\subsection{}
Let $A\subset V$ be a subset. We shall denote by $[A]\subset H=H_\Gamma$ the subgroup generated by $A$. We shall also denote by $A^\perp\subset H$
the normal subgroup generated by the set $V-A$.
Note that we do not exclude the case $A=\emptyset$; in that case $[A]=1$ and $A^\perp = H$.

\subsection{} \label{sub:fmaps}
Every subset $A\subset V$ determines a homomorphism $f_A: H\to H$ as follows. We define $f_A$ on the set of generators
$V\subset H$ by setting $f_A(v) =v$ for $v\in A$ and $f_A(v)=1$ for $v\in V-A$. For every relation $vw=wv$ of $H$, either (1) both
vertices $v, w\in A$ are in $A$, or (2) only one of the vertices $v, w$ lies in $A$, or (3) none of $v, w$ lies in $A$. In either case we have
$f_A(v)f_A(w)=f_A(w)f_A(v)$, which shows that the homomorphism $f_A$ is well defined.

For any two subsets $A, B\subset V$ one has
$$f_A\circ f_B =f_{A\cap B}=f_B\circ f_A ,$$
and $f_\emptyset =1$, $f_V={\rm id}_H$. Extending multiplicatively, the image of $f_A$ coincides with $[A]$ and moreover
$f_A(x)=x$ for any $x\in [A]$. In particular, if $A$ is a clique, then ${\rm Im}(f_A)=[A]$ is a free abelian group on the vertices in $A$.

 \begin{lemma}
 For any two subsets $A, B\subset V$ one has $[A]\cap [B]=[A\cap B]$.
 \end{lemma}
 \begin{proof}
 Obviously $[A\cap B]\subset [A]\cap [B]$; hence we only need to show that $[A]\cap [B]\subset [A\cap B]$. If $x\in [A]\cap [B]$ then
 $f_{A\cap B}(x) =f_A(f_B(x)) =x$ implying that $x\in {\rm Im} (f_{A\cap B})=[A\cap B]$.
 \end{proof}

 \begin{proposition}
For any set of cliques $C_1, C_2, \dots, C_r\subset V$ with empty intersection, $\cap_{i=1}^r C_i=\emptyset$, one has
$$\bigcap_{i=1}^r \left(g_i[C_i]g_i^{-1}\right) = \{1\},$$
for any collection of elements $g_1, \dots, g_r\in H$.
\end{proposition}
\begin{proof}
Let $x \in \bigcap_{i=1}^r \left(g_i[C_i]g_i^{-1}\right)$. Then $x=g_i c_i g_i^{-1}$ with $c_i \in [C_i]$ for $i=1,\ldots,r$.
Then, for all $i$ and $j$, we have $c_j = g_j^{-1}g_i c_i (g_j^{-1}g_i)^{-1}$ so that, by the discussion above, we have
$f_{C_i}(c_j)=c_i$ since each $[C_i]={\rm Im}(f_{C_i})$ is abelian. But then applying the equality $f_{C_i}(c_j)=c_i$ inductively $r$ gimes,
we obtain $f_{C_1}\circ f_{C_2} \circ \cdots\circ f_{C_r}(c_1) = c_1$.
Hence using the fact that
$$f_{C_1}\circ f_{C_2} \circ \cdots\circ f_{C_r}=f_{\cap_{i=1}^r C_i}$$
has image equal to $[\cap_{i=1}^r C_i]=\{1\}$ 
we see that $c_1=1$ 
and hence $x=1$.
\end{proof}

\begin{corollary}
Let $C_1, \dots, C_r\subset V$ be a set of cliques with empty intersection. Then the group
$$K=[C_1]\times [C_2]\times \dots\times [C_r]\subset H^r$$
satisfies the condition of Theorem \ref{GLO}, i.e. for any subgroup $L\subset H^r=H\times H\times \cdots\times H$
which is conjugate to the diagonal $\Delta\subset H^r$ one has $L\cap K=\{1\}$.
\end{corollary}

We have now shown that  $z_r(\Gamma)$ is given by a certain set of cliques $C_1,\ldots,C_r$ with empty total intersection and these in turn
determine a subgroup $K = [C_1]\times [C_2]\times \dots\times [C_r]\subset H^r$ with
$$\cd(K)=\sum_{i=1}^r |C_i|=z_r(\Gamma)$$
since each $[C_i]$ is free abelian. Applying Theorem \ref{GLO} then provides the lower bound $z_r(\Gamma) \leq \tc_r(H_\Gamma)$.


\section{The $\tc$-generating function}

\subsection{} For a group $H$ consider the following {\it $\tc$-generating function:}
\begin{eqnarray}
\mathcal F_H (x) = \sum_{r=1}^\infty \tc_{r+1}(H)\cdot x^r.
\end{eqnarray}
It is a formal power series whose coefficients are the integers $\tc_r(H)$.

\begin{example}\label{ex81}{\rm
Let $\mathcal H$ be the Higman's group as in Theorem \ref{higman}.
Then $\tc_r(\mathcal H)=2r$ for any $r$ and the $\tc$-generating function has the form
$$\mathcal F_{\mathcal H}(x) \, = \, \sum_{r=1}^\infty 2(r+1)x^r = \frac{2x(2-x)}{(1-x)^2}.$$
}
\end{example}
In this section we make the following observation:
\begin{theorem}\label{gen}
Let $H=H_\Gamma$ be a right angled Artin group. Then
the $\tc$-generating function $\mathcal F_H(x)$ is a rational function of the form
$$\mathcal F_H(x) =\frac{P_\Gamma(x)}{(1-x)^2},$$
where $P_\Gamma(x)$ is an integer polynomial with $P_\Gamma(1) =c(\Gamma)=\cd(H_\Gamma)$.
\end{theorem}

The proof of Theorem \ref{gen} given below uses the following lemmas in which we assume that
$H=H_\Gamma$ is a RAA group associated to a graph $\Gamma=(V, E)$. We abbreviate the notation $z_r(\Gamma)$ to $z_r$.


\begin{lemma}\label{without}
Suppose that $C_1, \dots, C_r\subset V$ is a sequence of cliques such that $\cap_{i=1}^r C_i = \emptyset$ and $\sum_{i=1}^r |C_i|=z_r$.
If additionally $\cap_{i=1}^{r-1}C_i=\emptyset$ then $z_r=z_{r-1}+c(\Gamma)$.
\end{lemma}

\begin{proof}
Note that  $z_{r-1}=\sum_{i=1}^{r-1}|C_i|$ and $|C_r|=c(\Gamma)$ since otherwise we would be able to increase the sum
$\sum_{i=1}^r |C_i|$ by replacing $C_1, \dots, C_{r-1}$ by a sequence realising $z_{r-1}$ and by replacing $C_r$ by a clique of size $c(\Gamma)$.  The result follows.
\end{proof}

Next we consider the case when $r$ is large enough.

\begin{lemma}\label{prev}
For $r> n=|V|$ one has $z_{r}=z_{r-1}+c(\Gamma)$.
\end{lemma}
\begin{proof}
Let $C_1, \dots, C_r\subset V$ be a sequence of cliques with $\cap_{i=1}^r C_i=\emptyset$ and $\sum_{i=1}^r |C_i|=z_r$.
Our statement will follow from Corollary \ref{without} once we know that the intersection of some $r-1$ cliques out of $C_1, \dots, C_r$ is empty.
Suppose the contrary, i.e. for any fixed $j=1, 2, \dots, r$ the intersection $\cap_{i\not=j}C_i\not=\emptyset$. Then we can find a point
$x_j\in \cap_{i\not=j}C_i$, i.e. $x_j\in C_i$ for any $i\not= j$. Clearly $x_j\not\in C_j$ since the total intersection of the cliques $C_i$ is empty. We obtain a sequence of points $x_1, \dots, x_r\in V$ and from our construction it is obvious that they are all pairwise distinct. But this contradicts our assumption $r>n=|V|$.
\end{proof}

\begin{proof}[Proof of Theorem \ref{gen}]
Using Lemma \ref{prev},
by induction  we find
\begin{eqnarray}\label{general}
z_r=z_n + (r-n)c(\Gamma) \quad \mbox{for}\quad r\ge n.
\end{eqnarray}
Using the equation $\tc_r(H_\Gamma)=z_r(\Gamma)$ given by Theorem \ref{GGY} we see that
\begin{eqnarray*}\F_H(x)&=& \sum_{r=1}^{n-1} z_{r+1}x^r + \sum_{r=n}^{\infty} z_{r+1}x^r \\
&=& \sum_{r=1}^{n-1} z_{r+1}x^r  + (z_n-n c(\Gamma))\sum_{r=n}^\infty x^r + c(\Gamma) \cdot \sum_{r=n}^\infty (r+1)x^r\\
&=& \sum_{r=1}^{n-1} z_{r+1}x^r + (z_n-n c(\Gamma))\frac{x^n}{1-x} + \frac{c(\Gamma)}{(1-x)^2} - c(\Gamma)\sum_{r=0}^{n-1}(r+1)x^r.
\end{eqnarray*}
The result follows since the first and the fourth terms are integer polynomials and the second and the third terms can be written as rational functions with denominator $(1-x)^2$.
\end{proof}

For an RAA group $H_\Gamma$, it is known that the maximum size $c(\Gamma)$ of a clique is equal to the cohomological dimension of
$H_\Gamma$ (which, in fact, is also the LS category of $K(H_\Gamma,1)$). Hence, we obtain the following.

\begin{theorem}
If $H_\Gamma$ is an RAA group and $|V_\Gamma|=n$, then
$$\tc_r(H_\Gamma) = \tc_n(H_\Gamma) + (r-n)\cd(H_\Gamma)$$
for $r \geq n$.
\end{theorem}

This is interesting because, while $\tc_r(X) - \tc_{r-1}(X) \leq 2\cat(X)$ holds for any $X$
(see \cite{BGRT}, Proposition 3.7), 
we see that for RAA groups we have
a precise description of the difference in terms of homological information about the group,
$$\tc_r(H_\Gamma) - \tc_{r-1}(H_\Gamma) = \cd(H_\Gamma) = \cat(K(H_\Gamma,1))$$
for $r > n$.

\begin{example}{\rm
Suppose that $H=H_\Gamma$ is a free group on $n\ge 2$ generators. In this case the graph $\Gamma$ has $n$ vertices and no edges.
We see that $z_r(\Gamma)=r$ for all $r\ge 2$. Hence
$$\mathcal F_H(x) = \frac{x(2-x)}{(1-x)^2}.$$
}
\end{example}

\begin{example}\label{exzn}{\rm
In the other extreme, suppose that $\Gamma$ is a complete graph on $n$ vertices. Then $H=\Z^n$ and $z_r(\Gamma)= (r-1)n$ for all $r\ge 2$. We obtain
$$\mathcal F_H (x) = \frac{nx}{(1-x)^2}.$$
}
\end{example}

\subsection{} Naturally, one may ask if the phenomenon of Theorem \ref{gen} holds in greater generality. More specifically, we ask
 {\it for which  finite CW-complexes $X$ the formal power series
$$\mathcal F_X (x) \, = \, \sum_{r=1}^\infty \tc_{r+1}(X)\cdot x^r$$
represents a rational function of the form
$$\frac{P_X(x)}{(1-x)^2}$$
where $P_X(x)$ is an integer polynomial satisfying $$P_X(1)=\cat(X)?$$}
The question above is equivalent to the statement that for all  $r$ large enough the following recurrence relation
$$\tc_{r+1}(X) = \tc_r(X) +\cat(X)$$
holds. It would be interesting to know the answer in the case when $X=K(G,1)$ for various classes of groups, say, for the class of hyperbolic groups.

Next we consider the following examples.

\begin{example}\label{exam:genfunc}\label{ex88}
{\rm
If $X=S^{2k+1}$, then we know that $\TC_r(X)=r-1$. Therefore, we have
$$\mathcal{F}_X(x)=\, x\cdot \sum_{r=0}^\infty (r+1)\cdot x^r = \frac{x}{(1-x)^2}.$$
For the even-dimensional sphere $X=S^{2k}$ we have $\tc_r(S^{2k}) = r$ and
$$\mathcal{F}_X(x) = \frac{x(2-x)}{(1-x)^2}.$$
}
\end{example}

\begin{example}\label{exam:liegp}
{\rm
If $X=G$ a compact Lie group, then we know that $\TC_r(G)=\cat(G^{r-1})$. Thus
$$\mathcal{F}_X(x)=\,  \sum_{r=1}^\infty \cat(G^{r})\cdot x^r.$$
Let's take a specific example where we know $\cat(G)$. Let $G=U(n)$. Then we know that $\cat(U(n))={\sf cup}_\Q(U(n))=n$
where ${\sf cup}_\Q$ denotes the rational cuplength. Note that rational cuplength obeys ${\sf cup}_\Q(X^r)=r\cdot {\sf cup}_\Q(X)$ for any $X$. Hence
we have $\cat(U(n)^r)=rn$ from the following:
$$rn=r\cdot {\sf cup}_\Q(X)={\sf cup}_\Q(X^r) \leq \cat(U(n)^r) \leq r\cdot \cat(U(n))=rn.$$
We then have
$$\mathcal{F}_{U(n)}(x) = \,\sum_{r=1}^\infty rn x^r = \frac{nx}{(1-x)^2}.$$
Note that $\mathcal{F}_{U(n)}(x)$ coincides with the $\tc$-generating function for the $n$-dimensional torus $\mathcal{F}_{T^n}(x)$, see example \ref{exzn}.
}
\end{example}


\begin{example}{\rm  Let $M^{2n}$ be a closed simply connected symplectic manifold. Then $\tc_r(M)=rn$ for any $r=1, 2, \dots$, see \cite{BGRT}, Corollary 3.15. Therefore,
$$\mathcal F_M(x) =\frac{nx(2-x)}{(1-x)^2}.$$

}
\end{example}

\begin{example}{\rm
Let $\Sigma_g$ denote the orientable surface of genus $g \geq 2$. Then, since $\tc_r(\Sigma_g) = 2r$ (see \cite[Proposition 3.2]{GGGHMR}),  
we obtain similarly to Example \ref{ex81},
\begin{align*}
\mathcal{F}_{\Sigma_g}(x) & =  \frac{2x(2-x)}{(1-x)^2}.
%
\end{align*}
In the cases $g=0$ and $g=1$ the answers are different, see Example
\ref{exzn} and Example \ref{ex88}.
}
\end{example}

\end{document}